\documentclass[12pt,leqno]{amsart}
\usepackage[total={6in,9in},
top=1in, left=1in, right=1in, bottom=1in]{geometry}
\usepackage{amsmath}
\usepackage[italian, english]{babel}
\usepackage{amsfonts}
\usepackage{amssymb}
\usepackage{dsfont}
\usepackage{mathrsfs}
\usepackage{graphicx}
\usepackage{float}
\usepackage{fancyhdr}
\usepackage{multirow}
\usepackage{multicol}
\usepackage{enumitem}
\usepackage{amsthm}
\usepackage{empheq}
\usepackage{cases}
\usepackage[all]{xy}
\usepackage{stmaryrd}
\usepackage[colorlinks, citecolor=blue, linkcolor=red]{hyperref}
\usepackage{mathrsfs}
\usepackage{tikz,tikz-cd}
\usepackage[font=small,labelfont=bf]{caption}
\usepackage[pagewise]{lineno}
\def\SS{{{\mathbb S}}}
\def\NN{{{\mathbb N}}}

\def\RR{{\mathbb R}}

\tikzset{
	subset/.style={
		draw=none,
		edge node={node [sloped, allow upside down, auto=false]{$\subset$}}},
	Subset/.style={
		draw=none,
		every to/.append style={
			edge node={node [sloped, allow upside down, auto=false]{$\subset$}}}
	}
}
\tikzset{
	labl/.style={anchor=south, rotate=90, inner sep=.50mm}
}

\newcommand\Tstrut{\rule{0pt}{2.9ex}}
\newcommand\Bstrut{\rule[-1.2ex]{0pt}{0pt}}
\newcommand\TBstrut{\Tstrut\Bstrut}

\newcommand{\erre}{\mathds{R}}

\newcommand{\ci}{\mathds{C}}

\newcommand{\del}[1]{\delta_{#1}}

\newcommand{\riem}{\operatorname{Riem}}
\newcommand{\ricc}{\operatorname{Ric}}
\newcommand{\weyl}{\operatorname{W}}
\newcommand{\diver}{\operatorname{div}}

\newcommand{\rp}{\erre\mathbb{P}}
\newcommand{\cp}{\ci\mathbb{P}}

\newcommand{\bigslant}[2]{{\raisebox{.0em}{$#1$}\left/\raisebox{-.0em}{$#2$}\right.}}
\newcommand{\KN}{\mathbin{\bigcirc\mspace{-15mu}\wedge\mspace{3mu}}}

\newcommand{\ra}{\rightarrow}

\newcommand{\longra}{\longrightarrow}

\newcommand{\Longra}{\Longrightarrow}

\newcommand{\pa}[1]{{\left(#1\right)}}                  
\newcommand{\sq}[1]{{\left[#1\right]}}                  
\newcommand{\abs}[1]{{\left|#1\right|}}                 

\newcommand{\eps}{\varepsilon}                           

\newcommand{\ol}[1]{\overline{#1}}
\renewcommand{\hat}[1]{\widehat{#1}}
\renewcommand{\tilde}[1]{\widetilde{#1}}

\newtheorem{theorem}{\textbf{Theorem}}[section]

\newtheorem{cor}[theorem]{\textbf{Corollary}}

\theoremstyle{remark}
\newtheorem{rem}[theorem]{\textbf{Remark}}

\numberwithin{equation}{section}
\title[]
{Rigidity of Einstein manifolds with positive Yamabe invariant}

\date{\today}


\begin{document}
	\maketitle
	
	\date{\today}
	
	\begin{center}
		\textsc{\textmd{L. Branca\footnote{Università degli Studi di Milano, Italy. Email: letizia.branca@unimi.it},
				G. Catino\footnote{Politecnico di Milano, Italy.
					Email: giovanni.catino@polimi.it.}, D. Dameno \footnote{Universit\`{a} degli Studi di Milano, Italy.
					Email: davide.dameno@unimi.it.}, P.
				Mastrolia\footnote{Universit\`{a} degli Studi di Milano, Italy.
					Email: paolo.mastrolia@unimi.it.}. }}
	\end{center}

	\begin{abstract}
		We provide optimal pinching results on closed Einstein manifolds with
		positive Yamabe invariant in any dimension, extending
		the optimal bound for the scalar curvature due to Gursky
		and LeBrun in dimension four. We also improve the known bounds
		of the Yamabe invariant \emph{via} the $L^{\frac{n}{2}}$-norm
		of the Weyl tensor for low-dimensional Einstein manifolds.
		Finally, we discuss some advances on an algebraic
		inequality involving the Weyl tensor for dimensions $5$ and $6$.
	\end{abstract}
	\section{Introduction and main results}
	The study of Riemannian functionals has proven to be widely important
	in the context of Riemannian Geometry and Geometric Analysis: indeed,
	many of the so-called \emph{special} (or \emph{canonical})
	Riemannian metrics arise as critical points of certain functionals,
	i.e. metrics which are solutions of the associated
	Euler-Lagrange equations. Given a closed smooth manifold $M$
	of dimension $n$, a classical
	example of such special cases is provided by Einstein
	metrics, which can be characterized as critical points of
	the celebrated \emph{Einstein-Hilbert functional}
	\begin{equation} \label{einhilb}
	\mathfrak{S}(g)=\mathrm{Vol}_g(M)^{-\frac{n-2}{n}}\int_M S_g\,d\mu_g,
	\end{equation}
	where $S_g$ and $\mathrm{Vol}_g(M)$ are, respectively,
	the scalar curvature and the
	volume with respect to the metric $g$. Other famous
	examples can be found if we consider the case $n=4$:
	for instance, the critical points of the
	\emph{Weyl functional}
	\[
	\mathfrak{W}(g)=\int_M\abs{\weyl_g}_g^{2}d\mu_g
	\]
	in dimension four are exactly
	the so-called \emph{Bach-flat metrics}, which
	have been intensively studied for many decades, due to their
	connection with General Relativity (\cite{bach}).
	The definition of the Weyl functional can be extended to higher
	dimensional cases, defining
	\begin{equation} \label{weylfun}
		\mathfrak{W}(g)=\int_M\abs{\weyl_g}_g^{\frac{n}{2}}d\mu_g,
	\end{equation}
	although Bach-flat metrics are no longer critical points if $n\neq 4$.
	It is worth to note that, for every $n$, \eqref{weylfun} is
	conformally invariant, i.e. it does not change under conformal
	changes of metric (see Section \ref{prelim} below): therefore,
	since, for any $n\geq 4$, Einstein metrics are also Bach-flat, this
	implies that a \emph{conformally Einstein metric}, i.e. a Riemannian
	metric whose conformal class contains an Einstein metric, is a critical
	point of \eqref{weylfun} as well, if $n=4$.
	
	Even though the existence of Einstein metrics requires, in general,
	strict topological conditions on $M$, it is always possible to
	find "non-obstructed" metrics by using \eqref{einhilb}: indeed,
	given a Riemannian metric $g$ on $M$ and its
	conformal class $[g]$, one can consider the so-called
	\emph{Yamabe invariant} $Y(M,[g])$, which is defined as the infimum
	of \eqref{einhilb} over the metrics $\tilde{g}\in[g]$.
	It is well-known that this infimum is always attained for every
	conformal class $[g]$ on $M$ and the metrics which actually
	achieve the minimum are \emph{constant scalar curvature metrics}
	(this is closely related to
	the so-called \emph{Yamabe problem}, see Section \ref{prelim}).
	
	Hypotheses on the sign of the Yamabe invariant may lead
	to surprising conclusions, especially in the four-dimensional case:
	for instance, a massive contribution
	was given by Gursky, who proved a sharp topological lower bound
	for the self-dual part of the Weyl functional, assuming
	the non-negativity of $Y(M,[g])$ and the existence of a positive
	eigenvalue for the intersection form of $M$ (\cite{gurskyann}).
	Later, this result was extended by the same author to half
	harmonic Weyl manifolds with positive Yamabe invariant
	(\cite{gurskyharmweyl}); moreover, strong rigidity results
	for four-manifolds with positive Yamabe invariant were proven in
	\cite{changguryang}, assuming additional curvature bounds.
	The same inequality obtained by Gursky
	was proven by LeBrun for conformal classes
	of symplectic type on a Del Pezzo surface, removing the hypothesis
	on the sign of $Y(M,[g])$ (\cite{lebrundelpezzo}).
	We also mention the result obtained by Chang, Gursky and Yang, who
	managed to prove that, given a closed Riemannian
	four-manifold $(M,g)$ with positive
	Yamabe invariant, there always exists a metric $\tilde{g}\in [g]$
	such that the Ricci tensor is strictly positive, provided that
	the integral of $\sigma_2(A)$,
	the second elementary symmetric function of
	the Schouten tensor, is positive (\cite{changgurskyyangann}).
	While, on one hand, all these results hold on closed four-manifolds,
	on the other hand some rigidity theorems can also be proven for
	compact manifolds of dimension four with boundary (see, for instance,
	\cite{catinondiaye}).
	
	In this paper, we are interested in sharp pinching results for
	Einstein closed manifolds of dimension $n\geq 4$
	with positive Yamabe invariant: in
	particular, we are interested in
	conformally invariant curvature
	inequalities of the form
	\[
	Y(M,[g])\leq A(n)\pa{\int_M\abs{\weyl}^{\frac{n}{2}}d\mu_g}^{\frac{2}{n}}.
	\]
	In dimension four, the optimal result was proven by Gursky and LeBrun
	for the self-dual part of the Weyl tensor, with constant $A(4)=\sqrt{6}$
	(\cite{gurskyharmweyl}, \cite{gurskylebrun}, see Remark \ref{dimfour}).
	As far as higher dimensional cases are concerned,
	Hebey and Vaugon proved that, for a Riemannian metric $g$
	on a closed manifold $M$ such that $[g]$ contains an Einstein metric
	or a locally conformally flat metric, either
	the Yamabe invariant, which is
	assumed to be positive, is bounded
	above by the $L^{\frac{n}{2}}$-norm of the tensor
	$Z=\weyl + \mathring{\ricc}$ or $(M,g)$ is isometric to a quotient
	of the standard sphere $\SS^n$ (\cite{hebeyvaugon}); their method
	relies on the classical Bochner-Weitzenb\"{o}ck formula and on the
	Yamabe-Sobolev
	inequality. A similar approach was used before by Singer to prove
	that, if $(M,g)$ is a $n$-dimensional Einstein manifold
	with positive scalar
	curvature, then $(M,g)$ is isometric to a quotient of the standard
	sphere, assuming that the $L^{\frac{n}{2}}$-norm
	of the Weyl curvature satisfies a
	pinching condition (\cite{singernorm}). The result due to Hebey and Vaugon
	was improved by the second and the fourth author,
	exploiting a method based on the Weitzenb\"{o}ck formula for the Weyl tensor
	(\cite{cmbook}). Moreover, we recall that Tran generalized and
	improved the previous bounds on closed manifolds with harmonic Weyl
	curvature (\cite{tran}).
	
	We point out that the aforementioned result are not sharp if $n>4$,
	meaning that the constants $A(n)$ are not the optimal ones.
	In this direction, a remarkable work due to Bour and Carron
	provides many answers about sharp pinching results for $n$-dimensional
	closed Riemannian manifolds with
	positive Yamabe invariant, under topological assumptions;
	the proofs are obtained \emph{via}
	an integral version of the Bochner-Weitzenb\"{o}ck formula on
	differential forms and
	a clever modification of the Yamabe invariant, which we will exploit
	as well throughout this paper
	(\cite{bourcarron}). In order to obtain better inequalities of
	the desired form, we rely on the classical Weitzenb\"{o}ck
	formula for the Weyl tensor, holding on every harmonic Weyl
	manifold, that is
	\begin{equation} \label{bochweitz}
		\dfrac{1}{2}\Delta\abs{\weyl}^2=\abs{\nabla\weyl}^2+\dfrac{2}{n}S\abs{
			\weyl}^2-2Q,
	\end{equation}
	where
	\begin{equation} \label{qdef}
		Q:=2W_{pqrs}W_{ptru}W_{qtsu}+\dfrac{1}{2}W_{pqrs}W_{pqtu}W_{rstu}
	\end{equation}
	(here, $W_{pqrs}$ are the components of the Weyl tensor with
	respect to a local orthonormal coframe).
	For some useful applications, see e.g.
	\cite{catinomastroliabochner,
	catmastrmontpunz, changguryang, derdz, gurskyharmweyl,
	hebeyvaugon, tran, wu}. The first step, which also is
	the main result of the paper, is to obtain a sharp upper bound
	of the Yamabe invariant with respect to a conformally invariant
	functional, involving $\weyl$ and $Q$. Namely, we are able
	to prove the following
 	\begin{theorem} \label{gianny}
		Let $(M,g)$ be a closed (conformally) Einstein manifold of dimension
		$n\geq 4$ with positive Yamabe invariant. Then, either
		$(M,g)$ is locally conformally flat (hence, a quotient
		of the round sphere) or, if $n\neq 5$ and $\weyl\not\equiv 0$,
		\begin{equation} \label{sharpinequality}
			Y(M,[g])\leq n\pa{\int_M
			\abs{Q}^{\frac{n}{2}}\abs{\weyl}^{-n}d\mu_g}^{\frac{2}{n}}.
		\end{equation}
		Moreover, equality holds in \eqref{sharpinequality}
		if and only if $(M,g)$ is locally symmetric.
		If
		$n=5$ and $\weyl\not\equiv 0$, then
		\begin{equation} \label{inequality5}
		Y(M,[g])\pa{1+\mathrm{Vol}_g^{-\frac{2}{5}}(M)\dfrac{\frac{1}{15}\int_M \abs{\weyl}\,d\mu_g}{\pa{\int_M \abs{\weyl}^{\frac{5}{3}}\,d\mu_g}^{\frac{3}{5}}}}\leq\frac{16}{3}
		\pa{\int_M\abs{Q}^{\frac{5}{2}}\abs{\weyl}^{-5}d\mu_g}
		^{\frac{2}{5}}
		\end{equation}
		and equality holds if and only if $(M,g)$ is locally symmetric.
		In particular,
		the following strict inequality holds:
		\begin{equation} \label{inequality5nonsharp}
				Y(M,[g]) < \dfrac{16}{3}\pa{\int_M
				\abs{Q}^{\frac{5}{2}}\abs{\weyl}^{-5}d\mu_g}^{\frac{2}{5}}.
		\end{equation}
	\end{theorem}
	As a consequence (see Section \ref{prelim}),
	we have the following lower bound for the $L^{\frac{n}{2}}$-norm of
	the Weyl curvature, improving the previous results in \cite{hebeyvaugon}
	for $5\leq n\leq 9$ and in \cite{tran} for $n=5,6$:
	\begin{cor} \label{corollary}
		Let $(M,g)$ be a closed (conformally) Einstein manifold of dimension
		$n\geq 4$ with positive Yamabe invariant. Then, either
		$(M,g)$ is locally conformally flat or
		\begin{equation} \label{hebeyimprove}
			Y(M,[g])\leq A(n)\pa{\int_M\abs{\weyl}^{\frac{n}{2}}
			d\mu_g}^{\frac{2}{n}},
		\end{equation}
		where $A(4)=\sqrt{6}$, $A(5)=\frac{64}{3\sqrt{10}}$,
		$A(6)=\sqrt{210}$ and
		$A(n)=\frac{5}{2}n$ for
		$n\geq 7$;
		if $n=5$, \eqref{hebeyimprove} is a strict inequality.
	\end{cor}
	\begin{rem} \label{dimfour}
		If $n=4$, the result is sharp: in fact, $\cp^2$ endowed with the
		Fubini-Study metric realizes the equality in \eqref{hebeyimprove}.
		We point out that Gursky and LeBrun provided the optimal pinching
		result, exploiting the peculiarities of $4$-dimensional manifolds
		(\cite{gurskyharmweyl}, \cite{gurskylebrun}):
		indeed, in this case the Weyl tensor can
		be regarded as a self-adjoint operator
		\[
		\mathcal{W}:\Lambda^2\longra\Lambda^2,
		\]
		where $\Lambda^2$ is the bundle of $2$-forms on $M$; moreover, in
		dimension four, the Hodge operator $\star$ induces the
		well-known decomposition
		\[
		\Lambda^2=\Lambda_+\oplus\Lambda_-,
		\]
		where $\Lambda_+$ (resp. $\Lambda_-$) is the subbundle of
		self-dual (resp. anti-self-dual) $2$-forms. This
		splitting leads to the decomposition of
		the Weyl operator into a self-dual and an anti-self-dual
		part, namely
		\[
		\mathcal{W}=\mathcal{W}_++\mathcal{W}_-,
		\]
		which, in turn, provides the well-known
		decomposition of the Weyl tensor
		\[
		\weyl=\weyl^++\weyl^-
		\]
		(for a detailed
		description, see for instance \cite{besse}, \cite{salamon}
		and \cite{singer}).
		Then, if $(M,g)$ is a closed Einstein manifold of dimension $4$
		with positive Yamabe invariant and such
		that $\weyl^+\not\equiv 0$, we have
		\begin{equation} \label{gurskypinch}
		Y(M,[g])\leq\sqrt{6}
		\pa{\int_M\abs{\weyl^+}^2d\mu_g}^{\frac{1}{2}},
		\end{equation}
		with equality if and only if $\nabla\weyl^+\equiv 0$; the same
		results holds if we replace $\weyl^+$ with $\weyl^-$. Note that
		both $\cp^2$ with the standard orientation and the Fubini-Study
		metric (which is a \emph{self-dual} manifold, i.e. $\weyl^-\equiv 0$)
		and $\SS^2\times\SS^2$ with the standard product metric
		realize the equality in \eqref{gurskypinch}. By the
		classification of irreducible symmetric spaces
		(\cite{cartan1}, \cite{cartan2}),
		we get that equality in \eqref{gurskypinch} is only realized
		by these manifolds, up to quotients.
	\end{rem}
	\begin{rem}
		If $n=5$, the constant $A(5)$ in \eqref{hebeyimprove} was
		$\frac{80}{3}$ in \cite{hebeyvaugon} and it is the same we obtained
		in \cite{tran}; in our case, the estimate is strict. If
		$n=6$, the constant $A(6)$ in \eqref{hebeyimprove} was
		$25$ in \cite{hebeyvaugon} and $15$ in \cite{tran}; note
		that $\sqrt{210}<15$. If $n\geq 7$, we recover the result
		in \cite{tran}, therefore improving the pinching in \cite{hebeyvaugon}
		for $7\leq n\leq 9$.
	\end{rem}
	The paper is organized as follows: in Section \ref{prelim}, we
	review some well-known facts about Riemannian manifolds and
	the Yamabe problem, recalling the classical definitions and
	some modifications of the Yamabe functional; after having fixed
	the notation,
	we proceed with the proof of the main results in Section \ref{proofthm}.
	Finally, in Section \ref{corebusiness}, we provide some remarks about
	inequality \eqref{qestimate}, exhibiting a lower bound for the optimal
	constant in dimension $6$ using a twistorial example; then, we
	describe a numerical approach simulating the Lagrange multiplier
	argument used to find the sharp constant in dimension $4$.
	
	\section*{Acknowledgements}
	The authors would like to express their gratitude to Dr. Giovanni Bocchi
	(Università degli Studi di Milano) for his precious contributions to the
	creation of the algorithm used in this paper. All authors are members of the
	Gruppo Nazionale per le Strutture Algebriche, Geometriche e loro Applicazioni
	(GNSAGA) of INdAM (Istituto Nazionale di Alta Matematica).
	
	\section{Preliminaries} \label{prelim}
	Let $(M,g)$ be a Riemannian manifold of dimension $n\geq 3$.
	It is well-known that
	the Riemann curvature tensor $\riem$ admits the
	decomposition
	\begin{equation} \label{riemdecomp}
		\riem=\weyl+\dfrac{1}{n-2}\ricc\KN g-
		\dfrac{S}{2(n-1)(n-2)}g\KN g,
	\end{equation}
	where $\weyl$, $\ricc$ and $S$ denote the Weyl curvature tensor,
	the Ricci tensor and the scalar curvature, respectively, and
	$\KN$ is the Kulkarni-Nomizu product (see, for instance,
	\cite{besse}).
	
	With respect to a local
	orthonormal frame,
	\eqref{riemdecomp} reads as
	\begin{equation} \label{riemdeclocal}
		R_{ijkt}=W_{ijkt}+\dfrac{1}{n-2}\pa{R_{ik}\del{jt}
		-R_{it}\del{jk}+R_{jt}\del{ik}-R_{jk}\del{it}}
		-\dfrac{S}{(n-1)(n-2)}\pa{\del{ik}\del{jt}-\del{it}\del{jk}},
	\end{equation}
	where $R_{ij}=R_{ikjk}$ and $S=R_{ii}$; throughout the paper, when
	dealing with local tensorial computations,
	we adopt Einstein's summation convention over
	repeated indices.
	
	When $(M,g)$ is an \emph{Einstein manifold}, i.e. when
	there exists $\lambda\in\RR$ such that $\ricc=\lambda g$,
	\eqref{riemdecomp} becomes
	\begin{equation*}
		\riem=\weyl+\dfrac{S}{2n(n-1)}g\KN g;
	\end{equation*}
	therefore, the curvature of any Einstein manifold of dimension
	$n\geq 4$ is encoded in the Weyl tensor and in the value of
	the scalar curvature $S$ (which is necessarily constant). Taking
	the squared norms of the tensors, we immediately obtain that,
	on any Einstein manifold,
	\begin{equation} \label{squarenormriem}
		\abs{\riem}^2=\abs{\weyl}^2+\dfrac{2S^2}{n(n-1)}.
	\end{equation}
	This equation will be constantly used in the first part of
	Section \ref{corebusiness}; we point out
	that our convention for the squared norm of a $(r,s)$-tensor field $T$
	is
	\[
	\abs{T}^2=T_{i_1...i_s}^{j_1...j_r}T_{i_1...i_s}^{j_1...j_r}.
	\]
	
	For a Riemannian manifold $(M,g)$ of dimension $n\geq 4$, the Weyl
	tensor is the totally trace-free part of $\riem$, while,
	on any $3$-dimensional Riemannian manifold, the Weyl tensor
	identically vanishes. One of the
	main properties of $\weyl$ resides in its behaviour under conformal
	deformations of the metric $g$: we recall that a \emph{conformal
	deformation} of $g$ is a new metric $\tilde{g}$ obtained by
	rescaling $g$ \emph{via} a smooth positive function $f$, i.e.
	\begin{equation} \label{confchange}
		\tilde{g}=f^2g,
	\end{equation}
	and that the \emph{conformal class} of $g$ is defined as
	\[
	[g]=\{\tilde{g}\in \mathcal{M}: \exists f\in C^{\infty}(M),
	f>0, \mbox{ s.t } \tilde{g}=f^2g\},
	\]
	where $\mathcal{M}$ denotes the space of smooth Riemannian metrics on $M$.
	For our purposes, it will be useful to choose $f=u^{2/(n-2)}$ throughout
	the paper, where $u\in C^{\infty}(M)$, $u>0$.
	
	Under the transformation \eqref{confchange}, the $(1,3)$-version
	of $\weyl$ does not change, i.e. the Weyl tensor is
	\emph{conformally invariant}. This important feature leads to
	a well-known characterization of a class of special Riemannian metrics:
	indeed, a Riemannian metric $g$ is \emph{locally conformally
	flat} if, for every $p\in M$, there exist an open neighborhood $U_p$
	of $p$ and a smooth positive function $f$ such that $(U,f^2g)$ is
	a flat submanifold of $M$. If $n\geq 4$, the celebrated
	Weyl-Schouten Theorem (see e.g. \cite{hertrich} or \cite{lee})
	states that this condition is equivalent to the vanishing of $\weyl$
	on $M$.
	In the next section we will need the transformed components of
	the $(0,4)$-version of $\weyl$,
	$\tilde{W}_{ijkt}$, which satisfy
	\begin{equation} \label{weylconf}
		u^{\frac{4}{n-2}}\tilde{W}_{ijkt}=W_{ijkt},
	\end{equation}
	and the expression of the transformed scalar curvature, i.e.
	\begin{equation} \label{uconf}
		u^{\frac{4}{n-2}}S_{\tilde{g}}=
		S_g-\dfrac{4(n-1)}{n-2}\dfrac{\Delta_gu}{u},
	\end{equation}
	where $\Delta_g$ is the Laplace-Beltrami operator of the metric $g$
	(for a full list of transformed curvature quantities, see
	e.g. \cite{cmbook}).
	
	An important class of metrics, generalizing the locally
	conformally flat ones, is represented by \emph{harmonic Weyl metrics},
	i.e. Riemannian metrics whose Weyl tensor satisfies
	\begin{equation*}
		\diver\weyl\equiv 0, \mbox{ on } M,
	\end{equation*}
	where $\diver$ is the divergence operator.
	It is well-known that
	this condition, for $n\geq 4$,
	is equivalent to the vanishing of the so-called
	\emph{Cotton tensor} $\operatorname{C}$ and, of course, it is
	satisfied by every locally conformally flat metric; moreover,
	a straightforward computation shows that all Einstein metrics
	are harmonic Weyl.
	
	As we mentioned in the Introduction, every harmonic Weyl manifold
	satisfies the Weitzenb\"{o}ck formula \eqref{bochweitz}: we highlight
	the fact that, in dimension four, the same formula holds for
	$\weyl^+$ and $\weyl^-$, and, in this case, one can prove that
	\[
	Q^{\pm}:=2W_{pqrs}^{\pm}W_{ptru}^{\pm}W_{qtsu}^{\pm}+
	\dfrac{1}{2}W_{pqrs}^{\pm}W_{pqtu}^{\pm}W_{rstu}^{\pm}=
	36\operatorname{det}_{\Lambda_{\pm}}\mathcal{W}_{\pm},
	\]
	where $\operatorname{det}_{\Lambda_{\pm}}\mathcal{W}_{\pm}$
	is the determinant of the linear operator $\mathcal{W}_{\pm}$
	from $\Lambda_{\pm}$ to itself (see Remark \ref{dimfour}).
	In general, we can find estimates for $Q$ in terms of
	$\abs{\weyl}^3$: namely,
	\begin{equation} \label{qestimate}
		\abs{Q}\leq C(n)\abs{\weyl}^3,
	\end{equation}
	with $C(4)=\frac{\sqrt{6}}{4}$, $C(5)=\frac{4}{\sqrt{10}}$,
	$C(6)=\frac{\sqrt{70}}{2\sqrt{3}}$ and $C(n)=\frac{5}{2}$ for
	$n\geq 7$. We recall that \eqref{qestimate} is an algebraic inequality,
	therefore an analogous estimate holds for every
	\emph{algebraic} Weyl curvature
	tensor $\weyl'$, i.e. a $(0,4)$-tensor which is totally trace
	free and satisfies
	the same symmetries as $\riem$.
	
	The constants in dimension $4$ and $6$ were obtained by
	Huisken (\cite{huisken}), exploiting a Lagrange multiplier argument
	and an idea due to Tachibana (\cite{tachibana}). We point out
	that $C(4)$ is the optimal constant, since equality in
	\eqref{qestimate} for $n=4$ is achieved by quotients
	of $\mathds{S}^4$, $\cp^2$ and $\RR\mathds{P}^4$.
	On the other hand, the constant $C(5)$ was obtained
	by Tran (\cite{tran}).
	Now, we recall again
	the definition of the classical \emph{Yamabe invariant}:
	\begin{equation} \label{yamabeinv}
		Y(M,[g])=\inf_{\tilde{g}\in[g]}\mathfrak{S}(g)=\inf_{\tilde{g}\in[g]}
		\mathrm{Vol}_{\tilde{g}}(M)^{-\frac{n-2}{n}}\int_M S_{\tilde{g}}
		d\mu_{\tilde{g}};
	\end{equation}
	as we briefly mentioned in the Introduction,
	the question of finding a metric $\hat{g}\in [g]$ that attains
	the minimum of \eqref{yamabeinv} is closely related to
	the so-called \emph{Yamabe problem}, i.e., the problem of finding
	a constant scalar curvature metric in any conformal class $[g]$ on
	any closed smooth manifold,
	whose final resolution was given by the joint efforts of
	Yamabe, Trudinger, Aubin and Schoen (for
	a detailed survey concerning the Yamabe problem, see for instance
	\cite{LeeParker}).
	It is well-known that a \emph{Yamabe minimizer}, i.e. a metric
	that attains the minimum of \eqref{yamabeinv}, is a metric with
	constant scalar curvature, whose sign is the same of $Y(M,[g])$.
	A fundamental tool for the understanding and the resolution of
	the Yamabe problem is given by the \emph{conformal Laplacian operator},
	that is
	\begin{equation*}
		\mathcal{L}_g=-\dfrac{4(n-1)}{n-2}\Delta_g+S_g;
	\end{equation*}
	we observe that $\mathcal{L}_gu$ represents the conformal
	change in \eqref{uconf}, with $\tilde{g}=u^{\frac{4}{n-2}}g$.
	In \cite{gurskyharmweyl},
	Gursky introduced a modified version of the conformal Laplacian,
	involving the Weyl tensor:
	\begin{equation} \label{lgursky}
		\mathcal{L}_g^t=-\dfrac{4(n-1)}{n-2}\Delta_g+S_g-t\abs{\weyl_g}_g,
	\end{equation}
	where $t\in\RR$. Starting from \eqref{lgursky}, he also defined a
	modified version of the Yamabe
	invariant, that is
	\begin{equation} \label{gurskymodifiedyamabe}
		\hat{Y}(M,[g])=\inf_{\tilde{g}\in[g]}
		\mathrm{Vol}_{\tilde{g}}(M)^{-\frac{n-2}{n}}\int_M \pa{S_{\tilde{g}}
		-t\abs{\weyl_{\tilde{g}}}_{\tilde{g}}}
		d\mu_{\tilde{g}};
	\end{equation}
	note that $\hat{Y}(M,[g])$ is indeed a conformal invariant,
	since (see \cite{gurskyharmweyl})
	\begin{equation} \label{gurskyconfchange}
	S_{\tilde{g}}-t\abs{\weyl_{\tilde{g}}}_{\tilde{g}}=
	u^{-\frac{n+2}{n-2}}\mathcal{L}_g^tu.
	\end{equation}
	Moreover, when $\hat{Y}(M,[g])\leq 0$,
	the modified Yamabe problem always admits a solution in every
	conformal class, as shown in \cite{gurskyharmweyl}, which means
	that, for every $[g]$, there always exists a metric $\hat{g}\in[g]$
	which attains the minimum of $\eqref{gurskymodifiedyamabe}$.
	We just mention that the remaining case $\hat{Y}(M,[g])>0$
	was studied by Itoh (\cite{itoh}). \\
	For the proof of Theorem \eqref{gianny}, we introduce a slightly
	different version of \eqref{gurskymodifiedyamabe}, depending on $S$, $Q$
	and $\weyl$: namely, we define the following
	\emph{modified Yamabe invariant}:
	\begin{equation} \label{modyamabeq}
		\ol{Y}^t(M,[g])=\inf_{\tilde{g}\in[g]}
		\mathrm{Vol}_{\tilde{g}}(M)^{-\frac{3}{5}}\int_M
		\pa{S_{\tilde{g}}-tQ_{\tilde{g}}\abs{\weyl_{\tilde{g}}}_{\tilde{g}}^{-2}}
		d\mu_{\tilde{g}},
	\end{equation}
	where $t\in\RR$. It can be shown that
	\begin{equation}\label{yamabemodconphi}
		\ol{Y}^t(M,[g])=\inf_{
			\stackrel{u\in C^{\infty}(M)}{u\neq 0}}
		\dfrac{\int_M u \mathfrak{L}_g^tu\,d\mu_g}{
			\pa{\int_Mu^{\frac{10}{3}}d\mu_g}^{\frac{3}{5}}},
	\end{equation}
	where $\mathfrak{L}^t$ is defined as
	\begin{equation} \label{lnostrodef}
	\mathfrak{L}_g^t=-\frac{4(n-1)}{n-2}\Delta_g+S_g-tQ_g\abs{\weyl_g}_g^{-2}.
	\end{equation}
	To show \eqref{yamabemodconphi}, we perform the aforementioned
	conformal change of the metric:
	then, by \eqref{weylconf} and \eqref{uconf}, the quantity
	$S-tQ\abs{\weyl}$ transforms as
	\begin{equation*}
		S_{\tilde{g}}-tQ_{\tilde{g}}\abs{\weyl_{\tilde{g}}}^{-2}=
		u^{-\frac{4}{n-2}}\sq{S-\dfrac{4(n-1)}{n-2}\dfrac{\Delta_gu}{u}-tQ_g\abs{\weyl_g}^{-2}}=u^{-\frac{
				n+2}{n-2}}{\mathfrak{L}}^tu,
	\end{equation*}
	while the conformal change for the volume form is given by
	\begin{equation*}
		\mu_{\tilde{g}}=
		u^{\frac{2n}{n-2}}
		\mu_g.
	\end{equation*}
	It follows that
	\begin{align*}
		\mathrm{Vol}_{\tilde{g}}(M)^{-\frac{n-2}{n}}\int_M
		S_{\tilde{g}}-t\abs{\weyl_{\tilde{g}}}_{\tilde{g}}d\mu_{\tilde{g}}&=
		\frac{\int_Mu^{-\frac{n+2}{n-2}}{\mathfrak{L}_g^t}u\cdot
		u^{\frac{2n}{n-2}}\,d\mu_g}{\pa{\int_Mu^{\frac{2n}{n-2}}d\,\mu_g}^{
		\frac{n-2}{n}}}=\frac{\int_Mu{\mathfrak{L}_g^t}u\,d\mu_g}{\pa{\int_Mu^
		{\frac{2n}{n-2}}\,d\mu_g}^{\frac{n-2}{n}}},
	\end{align*}
	which implies \eqref{yamabemodconphi}.
	
	
	\section{Proof of Theorem \ref{gianny}} \label{proofthm}
	\noindent
	For the sake of simplicity, we will omit to write the
	dependance from $g$, when it is not necessary.
	Assume that $\abs{\weyl}^2\not\equiv 0$ on $M$ and
	let
	\[
	\alpha=\dfrac{n-3}{2(n-1)}.
	\]
	We want to find integral estimates for a suitable
	second order operator applied to $\abs{\weyl}^{2\alpha}$:
	in order to do so, we
	exploit a strategy similar to the
	ones in
	\cite{bourcarron} and \cite{gurskyharmweyl} to deal with
	the points at which $\abs{\weyl}$ vanishes, where smoothness of
	$\abs{\weyl}^{2\alpha}$ fails.
	Let $\eps>0$ and
	\[
	f_\eps:=\pa{\abs{\weyl}^2+\eps^2}^{\alpha}:
	\]
	by a
	straightforward computation and \eqref{bochweitz}, we have that
	\begin{align} \label{laplfeps}
	\Delta f_\eps&=
	\alpha(\alpha-1)f_\eps^{1-\frac{2}{\alpha}}
	\abs{\nabla\abs{\weyl}^2}^2+
	\alpha f_\eps^{1-\frac{1}{\alpha}}\Delta\abs{\weyl}^2=\\
	&=4\alpha(\alpha-1)f_\eps^{1-\frac{2}{\alpha}}
	\abs{\weyl}^2\abs{\nabla\abs{\weyl}}^2+
	\alpha f_\eps^{1-\frac{1}{\alpha}}
	\pa{2\abs{\nabla\weyl}^2+\dfrac{4}{n}S\abs{\weyl}^2
	-4Q}\notag.
	\end{align}
	Now we exploit a refined Kato inequality for Einstein metrics (see
	\cite{bankasnak} for a proof), which reads as
	\begin{equation} \label{refkato}
		\abs{\nabla\abs{\weyl}}\leq\sqrt{\dfrac{n-1}{n+1}}\abs{\nabla\weyl};
	\end{equation}
	then, from \eqref{laplfeps} we deduce
	\[
		\Delta f_\eps\geq
		\alpha\sq{4(\alpha-1)
		\abs{\weyl}^2f_\eps^{1-\frac{2}{\alpha}}
		+\dfrac{2(n+1)}{n-1}f_\eps^{1-\frac{1}{\alpha}}}
		\abs{\nabla\abs{\weyl}}^2+
		4\alpha f_\eps^{1-\frac{1}{\alpha}}\pa{
		\dfrac{S}{n}\abs{\weyl}^2-Q}.
	\]
	Moreover, since,
	by definition, $f_\eps\geq\abs{\weyl}^{2\alpha}$, we get
	\[
	\Delta f_\eps\geq
	2\alpha\pa{2(\alpha-1)+\dfrac{n+1}{n-1}}
	\abs{\weyl}^2\abs{\nabla\abs{\weyl}}^2
	f_\eps^{1-\frac{2}{\alpha}}
	+4\alpha f_\eps^{1-\frac{1}{\alpha}}
	\pa{\dfrac{S}{n}\abs{\weyl}^2-Q},
	\]
	and, by our choice of $\alpha$, we conclude that
	\begin{equation} \label{laplestimate}
		\Delta f_\eps\geq 4\alpha
		\pa{\dfrac{S}{n}\abs{\weyl}^2-Q}f_\eps^{1-\frac{1}{\alpha}}.
	\end{equation}
	Now we adapt an idea due to Bour and Carron, defining the
	following operator:
	\begin{equation} \label{modifiedyamabeoperator}
		\mathscr{L}^\beta=-\dfrac{4(n-1)}{n-2}\Delta+\beta S
		-\beta nQ\abs{\weyl}^{-2}.
	\end{equation}
	By \eqref{laplestimate}, \eqref{modifiedyamabeoperator}
	and the definition of $f_\eps$,
	we obtain the estimate
	\begin{align*}
	f_\eps \mathscr{L}^\beta f_\eps&=
	-\dfrac{4(n-1)}{n-2}f_\eps\Delta f_\eps+\beta Sf_\eps^2
	-\beta nQ\abs{\weyl}^{-2}f_\eps^2\leq\\
	&\leq-\dfrac{8(n-3)}{n-2}\pa{\dfrac{S}{n}-Q\abs{\weyl}^{-2}}
	\abs{\weyl}^2 f_\eps^{2-\frac{1}{\alpha}}+
	\beta Sf_\eps^2
	-\beta nQ\abs{\weyl}^{-2}f_\eps^2=\\
	&=-\dfrac{8(n-3)}{n-2}\pa{\dfrac{S}{n}-Q\abs{\weyl}^{-2}}
	\pa{1-\eps^2f_\eps^{-\frac{1}{\alpha}}} f_\eps^2+
	\beta Sf_\eps^2
	-\beta nQ\abs{\weyl}^{-2}f_\eps^2=\\
	&=\pa{\beta-\dfrac{8(n-3)}{n(n-2)}}Sf_\eps^2
	+\pa{\dfrac{8(n-3)}{n-2}-\beta n}Q\abs{\weyl}^{-2}f_\eps^2+
	\dfrac{8(n-3)}{n-2}\eps^2\pa{\dfrac{S}{n}-Q\abs{\weyl}^{-2}}
	f_\eps^{2-\frac{1}{\alpha}}.
	\end{align*}
	Since, for $n=4$ and $n\geq 6$,
	\[
	\dfrac{8(n-3)}{n(n-2)}\leq 1,
	\]
	we can choose $\ol{\beta}=\frac{8(n-3)}{n(n-2)}$ in order to have
	\begin{equation*}
		f_\eps \mathscr{L}^{\ol{\beta}}
		f_\eps\leq\dfrac{8(n-3)}{n-2}
		\eps^2\pa{\dfrac{S}{n}-Q\abs{\weyl}^{-2}}f_\eps^{2-\frac{1}{\alpha}}.
	\end{equation*}
	Finally, since $2-\frac{1}{\alpha}<0$ for $n\geq 4$ and
	$f_\eps\geq\eps^{2\alpha}$, we get
	\begin{equation} \label{epsestimate}
	\eps^2f_\eps^{2-\frac{1}{\alpha}}\leq
	\eps^2\cdot\eps^{2\alpha\pa{2-\frac{1}{\alpha}}}=
	\eps^{4\alpha}\to 0, \mbox{ as } \eps\to 0.
	\end{equation}
	Now, we adapt a modification of the Yamabe invariant
	(again due to Bour and Carron, see \cite{bourcarron}),
	in the following way:
	\begin{equation} \label{yamabebour}
		\ol{Y}_g(\beta):=
		\inf_{\stackrel{\phi\in C^{\infty}(M)}{\phi\neq 0}}
		\dfrac{\int_M\phi \mathscr{L}^{\beta}\phi\, d\mu_g}{
		\pa{\int_M\phi^{\frac{2n}{n-2}}d\mu_g}^{\frac{n-2}{n}}}=
		\inf_{\stackrel{\phi\in C^{\infty}(M)}{\phi\neq 0}}
		\dfrac{\int_M\pa{\frac{4(n-1)}{n-2}\abs{\nabla\phi}^2
			+\beta S\phi^2-\beta nQ\abs{\weyl}^{-2}\phi^2}d\mu_g}{
			\pa{\int_M\phi^{\frac{2n}{n-2}}d\mu_g}^{\frac{n-2}{n}}},
	\end{equation}
	for all $\beta\geq 0$. Observe that
	$\ol{Y}_g(1)$ is equal to \eqref{modyamabeq} with
	$t=n$ and, since $M$ is closed,
	$\ol{Y}_g(0)=0$: moreover, \eqref{yamabebour} is the infimum
	of affine functions of $\beta$, therefore it is concave and,
	for $\beta\in [0,1]$,
	\[
	(1-\beta)\ol{Y}_g(0)+\beta\ol{Y}_g(1)\leq \ol{Y}_g(\beta),
	\]
	which implies that
	\begin{equation} \label{concave}
		\beta\ol{Y}^n(M,[g])\leq\ol{Y}_g(\beta).
	\end{equation}
	Note that, by \eqref{epsestimate} and the definition of
	$f_\eps$, we get
	\begin{align*}
	\ol{Y}_g(\ol{\beta})&\leq\dfrac{\int_M f_\eps \mathscr{L}^{\ol{\beta}}f_\eps\,
	d\mu_g}{\pa{\int_M f_\eps^{\frac{2n}{n-2}}\,d\mu_g}^{\frac{n-2}{n}
	}}
	\leq
	\dfrac{8(n-3)}{n-2}
	\dfrac{\int_M
	\pa{\frac{S}{n}-Q\abs{\weyl}^{-2}}f_\eps^{2-\frac{1}{\alpha}}\,
		d\mu_g}{\pa{\int_M f_\eps^{\frac{2n}{n-2}}\,d\mu_g}^{\frac{n-2}{n}}}
	\eps^2\leq\\
	&\leq\dfrac{8(n-3)}{n-2}
	\dfrac{\int_M
		\pa{\frac{S}{n}-Q\abs{\weyl}^{-2}}
		\eps^2f_\eps^{2-\frac{1}{\alpha}}\,
		d\mu_g}{\pa{\int_M u^{\frac{2n}{n-2}}\,d\mu_g}^{\frac{n-2}{n}}}
	 \xrightarrow[\eps\to 0]{}0,
	\end{align*}
	which implies that
	\begin{equation*}
		\ol{Y}_g(\ol{\beta})\leq 0.
	\end{equation*}
	Since $n\neq 5$, $\ol{\beta}\in(0,1)$, therefore,
	by \eqref{concave}, we obtain
	\[
	\bar{\beta}\ol{Y}^n(M,[g])\leq \ol{Y}_g(\ol{\beta}),
	\]
	which implies that
	\begin{equation*}
		\ol{Y}^n(M,[g])\leq 0.
	\end{equation*}
	Now, following the same line of reasoning in \cite{gurskyharmweyl}
	(i.e., adapting the argument of \cite[Proposition 4.4]{LeeParker}),
	we know that there exists a unique metric
	$\hat{g}\in [g]$ which attains the minimum
	of \eqref{modyamabeq}, that is
	\[
	\mathrm{Vol}_{\hat{g}}(M)^{-\frac{2}{n}}\ol{Y}^n(M,[g])=
	\hat{S}-n\hat{Q}\abs{\hat{\weyl}}^{-2},
	\]
	where $\hat{S}$, $\hat{Q}$ and $\hat{\weyl}$ are relative to
	the metric $\hat{g}$. Then, since
	$\ol{Y}(M,[g])\leq 0$ and $g$ is an Einstein metric
	(therefore, it attains the minimum of $Y(M,[g])$), since
	$\hat{g}\in[g]$, we have that
	\[
	Y(M,[g])=\mathrm{Vol}_{g}(M)^{-\frac{n-2}{n}}\int_MS\,d\mu_g\leq
	\mathrm{Vol}_{\hat{g}}(M)^{-\frac{n-2}{n}}\int_M\hat{S}\,d\mu_{\hat{g}}
	\leq \mathrm{Vol}_{\hat{g}}(M)^{-\frac{n-2}{n}}\int_M
	n\hat{Q}\abs{\hat{\weyl}}^{-2}d\mu_{\hat{g}}.
	\]
	By H\"{o}lder inequality, we obtain
	\[
	\int_M
	\hat{Q}\abs{\hat{\weyl}}^{-2}d\mu_{\hat{g}}\leq
	\pa{\int_M\abs{\hat{Q}}^{\frac{n}{2}}\abs{\hat{\weyl}}^{-n}
		d\mu_{\hat{g}}}^{\frac{2}{n}}
	\mathrm{Vol}_{\hat{g}}(M)^{\frac{n-2}{n}},
	\]
	which implies that
	\[
	Y(M,[g])\leq n\pa{\int_M\abs{\hat{Q}}^{\frac{n}{2}}\abs{\hat{\weyl}}^{-n}
		d\mu_{\hat{g}}}^{\frac{2}{n}}.
	\]
	Now, by formula \eqref{weylconf},
	we know that the right-hand side of the previous inequality
	is conformally invariant; therefore,
	\[
	Y(M,[g])\leq n\pa{\int_M\abs{Q}^{\frac{n}{2}}\abs{\weyl}^{-n}
		d\mu_g}^{\frac{2}{n}}
	\]
	and inequality \eqref{sharpinequality} is proven.
	
	Now, suppose that the equality in \eqref{sharpinequality} holds: then,
	\[
	\int_M\hat{S}\,d\mu_{\hat{g}}=
	\int_M
	n\abs{\hat{Q}}\abs{\hat{\weyl}}^{-2}d\mu_{\hat{g}},
	\]
	which immediately implies that, by definition,
	$\ol{Y}^n(M,[g])=0$. Moreover,
	since
	\[
	Y(M,[g])=\mathrm{Vol}_{\hat{g}}(M)^{-\frac{n-2}{n}}
	\int_M \hat{S}\, d\mu_{\hat{g}},
	\]
	we observe that $\hat{g}$ attains the minimum $Y(M,[g])$ and,
	therefore,
	$\hat{g}$ is a solution of the Yamabe problem in $[g]$,
	which implies that $\hat{S}$ is constant:
	hence, since $g$ is an Einstein
	metric in $[g]$,
	we can exploit a well-known result due to Obata (\cite{obata})
	in order to conclude that $\hat{g}=g$ and, as a consequence,
	$\hat{S}=S$, $\hat{Q}=Q$ and $\hat{\weyl}=\weyl$.
	This implies that
	\begin{equation} \label{SQ}
	S-nQ\abs{\weyl}^{-2}=0 \Longra
	Q=\dfrac{S}{n}\abs{\weyl}^2;
	\end{equation}
	therefore, integrating \eqref{bochweitz}, we get
	\[
	0=\int_M\pa{\abs{\nabla\weyl}^2+\dfrac{2}{n}S\abs{\weyl}^2-2Q\,} d\mu_g=
	\int_M\abs{\nabla\weyl}^2d\mu_g,
	\]
	which implies that $\abs{\nabla\weyl}^2\equiv 0$ on $M$, i.e.
	$(M,g)$ is locally symmetric. The converse is trivial, since,
	by \eqref{bochweitz}, we immediately obtain \eqref{SQ}.
	
	Now, for $n=5$, we consider again \eqref{modyamabeq}, choosing
	$t=\frac{16}{3}$; also, in order to simplify the notation,
	we write $\ol{Y}(M,[g])=\ol{Y}^{\frac{16}{3}}(M,[g])$
	and $\mathfrak{L}=\mathfrak{L}^{{\frac{16}{3}}}$.
	Exploiting the same technique used for $n\neq 5$ and noting
	that $\alpha=\frac{1}{4}$ in this case,
	we obtain the estimate
	\begin{equation}\label{estimateLt}
		f_\eps{\mathfrak{L}}f_\eps\leq -\dfrac{1}{15}Sf_\eps^2
		+
		\dfrac{16}{3}\eps^2\pa{\dfrac{S}{5}-Q\abs{\weyl}^{-2}}
		f_\eps^{-2}.
	\end{equation}
	Note that, since $f_\eps>\eps^{\frac{1}{2}}$ for $n=5$,
	we can deduce
	\begin{align}\label{eps5}
		\eps^2f_\eps^{-2}\leq \eps\ra 0, \mbox{ as }\eps\ra 0.
	\end{align}
	Hence, by \eqref{estimateLt} and the definition of $f_\eps$, we deduce
	\begin{align}\label{ineqYmod5}
		\ol{Y}(M,[g])&\leq \dfrac{\int_M f_\eps {\mathfrak{L}}f_\eps\,
			d\mu_g}{\pa{\int_M f_\eps^{\frac{10}{3}}\,d\mu_g}^{\frac{3}{5}
		}}
		\leq
		-\dfrac{\frac{1}{15}S\int_M f_\eps^2\,d\mu_g}{\pa{\int_M f_\eps^{\frac{10}{3}}\,d\mu_g}^{\frac{3}{5}}}+
		\dfrac{16}{3}
		\dfrac{\int_M
			\pa{\frac{S}{5}-Q\abs{\weyl}^{-2}}f_\eps^{-2}\,
			d\mu_g}{\pa{\int_M f_\eps^{\frac{10}{3}}\,d\mu_g}^{\frac{3}{5}}}
		\eps^2.
	\end{align}
	Note that, since $(\abs{\weyl}^2+1)^{\frac{1}{4}}$ is integrable on $M$, by the dominated convergence theorem and \eqref{eps5} we obtain
	\begin{align*} 	
	   -\dfrac{\frac{1}{15}S\int_M f_\eps^2\,d\mu_g}{\pa{\int_M f_\eps^{\frac{10}{3}}\,d\mu_g}^{\frac{3}{5}}}+
	   \dfrac{16}{3}
	   \dfrac{\int_M
	   	\pa{\frac{S}{5}-Q\abs{\weyl}^{-2}}f_\eps^{-2}\,
	   	d\mu_g}{\pa{\int_M f_\eps^{\frac{10}{3}}\,d\mu_g}^{\frac{3}{5}}}
	   \eps^2 \xrightarrow[\eps\to 0]{}
	   	-\dfrac{\frac{1}{15}S\int_M \abs{\weyl}\,d\mu_g}{{\pa{\int_M \abs{\weyl}^{\frac{5}{3}}\,d\mu_g}^{\frac{3}{5}}}},
	\end{align*}
	implying
	\[
	 \ol{Y}(M,[g])\leq 	-\dfrac{\frac{1}{15}S\int_M\abs{\weyl}\,d\mu_g}{{\pa{\int_M \abs{\weyl}^{\frac{5}{3}}\,d\mu_g}^{\frac{3}{5}}}}\leq0.
	\]
	Arguing as in the case $n\neq 5$, we have that there exists a unique
	metric $\hat{g}\in [g]$ attaining the minimum of the modified Yamabe
	invariant $\ol{Y}(M,[g])$, i.e.
	\[\mathrm{Vol}_{\hat{g}}(M)^{-\frac{2}{5}}\ol{Y}(M,[g])=
	\hat{S}-\frac{16}{3}\hat{Q}\abs{\hat{\weyl}}^{-2},
	\]
	where $\hat{S}$, $\hat{Q}$ and $\hat{\weyl}$ are relative to
	the metric $\hat{g}$. Moreover, since $g$ attains the minimum of
	$Y(M,[g])$ and $\hat{g}\in[g]$, we have the inequalities
	\begin{align*}
	Y(M,[g])&=\mathrm{Vol}_{g}(M)^{-\frac{3}{5}}\int_MS\,d\mu_g\leq
	\mathrm{Vol}_{\hat{g}}(M)^{-\frac{3}{5}}\int_M\hat{S}\,d\mu_{\hat{g}}
	\leq\\
	&\leq \mathrm{Vol}_{\hat{g}}(M)^{-\frac{3}{5}}\int_M
	\frac{16}{3}\hat{Q}\abs{\hat{\weyl}}^{-2}d\mu_{\hat{g}}-
	\dfrac{\frac{1}{15}S\int_M \abs{\weyl}\,d\mu_g}{\pa{\int_M \abs{\weyl}^{\frac{5}{3}}\,d\mu_g}^{\frac{3}{5}}}.
	\end{align*}
	By H\"{o}lder inequality, we deduce
	\[Y(M,[g])\leq \frac{16}{3}
	\pa{\int_M\abs{\hat{Q}}^{\frac{5}{2}}\abs{\hat{\weyl}}^{-5}d\mu_{\hat{g}}}
	^{\frac{2}{5}}-\dfrac{\frac{1}{15}S\int_M \abs{\weyl}\,d\mu_g}{\pa{\int_M \abs{\weyl}^{\frac{5}{3}}\,d\mu_g}^{\frac{3}{5}}},
	\]
	that is, since $S=\mathrm{Vol}_g(M)^{\frac{2}{5}}Y(M,[g])$,
	\begin{align}\label{YQW5}
		Y(M,[g])\pa{1+\mathrm{Vol}_g(M)^{-\frac{2}{5}}\dfrac{\frac{1}{15}\int_M \abs{\weyl}\,d\mu_g}{\pa{\int_M \abs{\weyl}^{\frac{5}{3}}\,d\mu_g}^{\frac{3}{5}}}}\leq \frac{16}{3}
		\pa{\int_M\abs{\hat{Q}}^{\frac{5}{2}}\abs{\hat{\weyl}}^{-5}d\mu_{\hat{g}}}
		^{\frac{2}{5}}=\frac{16}{3}
		\pa{\int_M\abs{Q}^{\frac{5}{2}}\abs{\weyl}^{-5}d\mu_g}
		^{\frac{2}{5}}.
	\end{align}
	It follows that inequality \eqref{inequality5} holds.
	If equality in \eqref{inequality5} is attained, then
	\begin{align*}
		\mathrm{Vol}_{\hat{g}}(M)^{-\frac{3}{5}}
		\int_M\hat{S}\,d\mu_{\hat{g}}= \mathrm{Vol}_{\hat{g}}(M)^{-\frac{3}{5}}\int_M
		\frac{16}{3}\abs{\hat{Q}}\abs{\hat{\weyl}}^{-2}d\mu_{\hat{g}}-
		\dfrac{\frac{1}{15}S\int_M \abs{\weyl}\,d\mu_g}{\pa{\int_M \abs{\weyl}^{\frac{5}{3}}\,d\mu_g}^{\frac{3}{5}}},
	\end{align*}
	which implies
	\begin{align}\label{eqn5yamabe}
		\ol{Y}(M,[g])=-\dfrac{\frac{1}{15}S\int_M \abs{\weyl}\,d\mu_g}{\pa{\int_M \abs{\weyl}^{\frac{5}{3}}\,d\mu_g}^{\frac{3}{5}}}.
	\end{align}
	Moreover, as for the case $n\neq 5$, $\hat{g}$ is a solution of the Yamabe problem, implying again that $\hat{g}=g$.
	Note that equality holds in H\"older's estimate: therefore,
	$Q\abs{\weyl}^{-2}$ is constant on $M$. Moreover,
	$Q\abs{\weyl}^{-2}\neq 0$ on $M$, otherwise, integrating
	\eqref{bochweitz}, we would conclude that $(M,g)$ is locally
	conformally flat, which contradicts our hypothesis.
	Now, since $\abs{\weyl}\neq 0$ on $M$,
	we can repeat the initial argument of the proof
	replacing $f_\eps$ with $\abs{\weyl}^{\frac{1}{2}}$ in order to
	get
	\begin{equation} \label{lestimate5}
		\mathfrak{L}\abs{\weyl}^{\frac{1}{2}}\leq
		-\dfrac{S}{15}\abs{\weyl}^{\frac{1}{2}};
	\end{equation}
	therefore, by \eqref{lestimate5}, we deduce
	\[
	\ol{Y}(M,[g])\leq\dfrac{\int_M\abs{\weyl}^{\frac{1}{2}}
		\mathfrak{L}\abs{\weyl}^{\frac{1}{2}}d\mu_g}{\pa{\int_M\abs{\weyl}^
		{\frac{5}{3}}}^{\frac{3}{5}}}\leq
		-\dfrac{\frac{1}{15}S\int_M \abs{\weyl}\,d\mu_g}{\pa{\int_M \abs{\weyl}^{\frac{5}{3}}\,d\mu_g}^{\frac{3}{5}}},
	\]
	and, since \eqref{eqn5yamabe} holds,
	we obtain
	\[
	\int_M\abs{\weyl}^{\frac{1}{2}}\mathfrak{L}\abs{\weyl}^{\frac{1}{2}}d\mu_g=
	-\dfrac{S}{15}\int_M\abs{\weyl}\, d\mu_g.
	\]
	By the previous equality and \eqref{lestimate5}, it is easy to observe
	that we must have a pointwise equality, namely
	\begin{equation} \label{equality5}
		\mathfrak{L}\abs{\weyl}^{\frac{1}{2}}=
		-\dfrac{S}{15}\abs{\weyl}^{\frac{1}{2}}.
	\end{equation}
	Using \eqref{lnostrodef} and integrating \eqref{equality5}, we conclude
	\[
	\pa{S-5Q\abs{\weyl}^{-2}}\int_M\abs{\weyl}^{\frac{1}{2}}d\mu_g
	=0;
	\]
	since $\abs{\weyl}\neq 0$ on $M$, the claim follows by \eqref{bochweitz}.
	
	Conversely, if $(M,g)$ is locally symmetric, equality in
	\eqref{inequality5} follows by observing that $\abs{\weyl}$ is
	a constant function and by using
	\eqref{bochweitz} and \eqref{yamabeinv}.
	\qed
	
 \begin{proof}[Proof of Corollary \ref{corollary}]
	Recall that, for $n\geq 4$,
	inequality \eqref{qestimate} holds with the following
	constants
	\[
	C(4)=\dfrac{\sqrt{6}}{4}, \mbox{ } C(5)=\frac{4}{\sqrt{10}}, \mbox{ } C(6)=\frac{\sqrt{70}}{2\sqrt{3}},
	\mbox{ } C(n)=\frac{5}{2} \mbox{ for } n\geq 7.
	\]
	Therefore, using \eqref{qestimate} in \eqref{sharpinequality}
	with these constants, the claim is proven.
\end{proof}
\section{Remarks on the sharp estimate for $Q$} \label{corebusiness}
As we mentioned in Section \ref{prelim}, Huisken (\cite{huisken}) exploited
a standard Lagrange multiplier in order to obtain a sharp constant
in the estimate \eqref{qestimate}: indeed, using the decomposition
described in Remark \ref{dimfour} and pointwise diagonalizing both
$\mathcal{W}_+$ and $\mathcal{W}_-$, the Lagrange multiplier
problem reduces to solving a system of six polynomial equations in the
eigenvalues of $\mathcal{W}_{\pm}$.
As we observed before, Theorem \ref{gianny} and \eqref{qestimate}
in dimension four partially recover the well-known pinching result
due to Gursky and LeBrun (\cite{gurskyharmweyl}, \cite{gurskylebrun}).
As far as higher dimensional cases are concerned, the problem of
finding the optimal constant $A(n)$ in
\begin{equation} \label{yamabepinch}
Y(M,[g])\leq A(n)\pa{\int_M\abs{\weyl}^{\frac{n}{2}}}^{\frac{2}{n}}
\end{equation}
is still open, also due to the fact that Einstein manifolds are far
less understood in these cases than in the four-dimensional one. In view
of Theorem \ref{gianny}, it seems apparent that this problem is
closely related to the existence of a sharp constant in the
estimate \eqref{qestimate}, since, in this case, the optimal pinching
\eqref{yamabepinch} would be a straightforward consequence of
\eqref{sharpinequality}.

In order to improve the investigation on the best constants in
\eqref{qestimate} and \eqref{yamabepinch}, it is natural to check
the most classical examples, such as locally symmetric, irreducible
manifolds $M=G/K$, which also happen to be Einstein.
Recall that, by \eqref{bochweitz}, if
$(M,g)$ is locally symmetric Einstein, then
\begin{equation} \label{qlocsymm}
Q=\dfrac{S}{n}\abs{\weyl}^2;
\end{equation}
therefore, if $C_M\in\RR$ is such that $Q=C_M\abs{\weyl}^3$ on $M$,
then
\[
C_M=\dfrac{S}{n\abs{\weyl}}.
\]
Also, let us denote $A_M\in\RR$ the constant such that
$S=A_M\abs{\weyl}$ on $M$.
By Cartan's classification
of classical Riemannian symmetric spaces
and some curvature results contained
in \cite[Table III]{grayvanhecke}, using \eqref{squarenormriem}
and \eqref{qlocsymm} we are able to describe these
cases as follows:\\[10pt]

\begin{table}[H]
	\caption{}
\begin{center}
	\begin{tabular}{|p{2cm} |p{2.5cm} |p{1.5cm} |p{1.5cm} |p{1.5cm} |p{1.5cm} |p{1.5cm} |}
		\hline
		\multicolumn{7}{|c|}{Classical $5$-dimensional symmetric spaces}\\
		\hline
		$G$ & $K$ & $S$ & $\abs{\weyl}^2$ & $Q$ & $C_M$ & $A_M$\TBstrut\\
		\hline
		$SU(3)$ & $SO(3)$ & $30$ & $210$ & $1260$ & $\frac{\sqrt{210}}{35}$ &
		$\frac{\sqrt{210}}{7}$ \TBstrut\\
		\hline
	\end{tabular}
\label{5classical}
\end{center}
\end{table}
\begin{table}[H]
	\caption{}
\begin{center}
\begin{tabular}{|p{2cm} |p{2.5cm} |p{1.5cm} |p{1.5cm} |p{1.5cm} |p{1.5cm} |p{1.5cm} |}
	\hline
	\multicolumn{7}{|c|}{Classical $6$-dimensional symmetric spaces}\\
	\hline
	$G$ & $K$ & $S$ & $\abs{\weyl}^2$ & $Q$ & $C_M$ & $A_M$\TBstrut\\
	\hline
	$SO(4)$ & $\{I_4\}$ & $24$ & $\frac{288}{5}$ & $\frac{1152}{5}$ &
	$\frac{\sqrt{10}}{6}$ & $\sqrt{10}$\TBstrut\\
	$SO(5)$ & $SO(2)\times SO(3)$ & $18$ & $\frac{312}{5}$ &
	$\frac{936}{5}$ & $\frac{\sqrt{390}}{52}$ &
	$\frac{3\sqrt{390}}{26}$\TBstrut\\
	$U(4)$ & $U(1)\times U(3)$ & $24$ & $\frac{288}{5}$ & $\frac{1152}{5}$ &
	$\frac{\sqrt{10}}{6}$ & $\sqrt{10}$\TBstrut\\
	$SO(6)$ & $U(3)$ & $24$ & $\frac{288}{5}$ & $\frac{1152}{5}$ &
	$\frac{\sqrt{10}}{6}$ & $\sqrt{10}$\TBstrut\\
	$Sp(2)$ & $U(2)$ & $36$ & $\frac{1248}{5}$ & $\frac{7488}{5}$
	& $\frac{\sqrt{390}}{52}$ & $\frac{3\sqrt{390}}{26}$\TBstrut\\
	\hline
\end{tabular}
\label{shu}
\end{center}
\end{table}
\begin{table}[H]
	\caption{}
\begin{center}
	\begin{tabular}{|p{2cm} |p{2.5cm} |p{1.5cm} |p{1.5cm} |p{1.5cm} |p{1.5cm} |p{1.5cm} |}
		\hline
		\multicolumn{7}{|c|}{Classical $8$-dimensional symmetric spaces}\\
		\hline
		$G$ & $K$ & $S$ & $\abs{\weyl}^2$ & $Q$ & $C_M$ & $A_M$\TBstrut\\
		\hline
		$SU(3)$ & $\{I_3\}$ & $96$ & $\frac{6096}{7}$ & $\frac{73152}{7}$ &
		$\frac{\sqrt{2667}}{127}$ & $8\frac{\sqrt{2667}}{127}$\TBstrut\\
		$SO(6)$ & $SO(2)\times SO(4)$ & $32$ & $\frac{864}{7}$ &
		$\frac{3456}{7}$ & $\frac{\sqrt{42}}{18}$ &
		$\frac{4\sqrt{42}}{9}$\TBstrut\\
		$U(4)$ & $U(2)\times U(2)$ & $32$ & $\frac{864}{7}$ & $\frac{3456}{7}$ &
		$\frac{\sqrt{42}}{18}$ & $\frac{4\sqrt{42}}{9}$\TBstrut\\
		$U(5)$ & $U(1)\times U(4)$ & $40$ & $\frac{720}{7}$ & $\frac{3600}{7}$ &
		$\frac{\sqrt{35}}{12}$ & $\frac{2\sqrt{35}}{3}$\TBstrut\\
		$Sp(3)$ & $Sp(1)\times Sp(2)$ & $64$ & $\frac{1888}{7}$ & $\frac{15104}{7}$
		& $\frac{\sqrt{826}}{59}$ & $\frac{8\sqrt{826}}{59}$\TBstrut\\
		\hline
	\end{tabular}
\label{tab8}
\end{center}
\end{table}
\begin{table}[H]
	\caption{}
	\begin{center}
	\begin{tabular}{|p{2cm} |p{2.5cm} |p{1.5cm} |p{1.5cm} |p{1.5cm} |p{1.5cm} |p{1.5cm} |}
	\hline
	\multicolumn{7}{|c|}{Classical $9$-dimensional symmetric spaces}\\
	\hline
	$G$ & $K$ & $S$ & $\abs{\weyl}^2$ & $Q$ & $C_M$ & $A_M$\TBstrut\\
	\hline
	$SO(6)$ & $SO(3)\times SO(3)$ & $36$ & $180$ & $720$ &
	$\frac{2\sqrt{5}}{15}$ & $\frac{6\sqrt{5}}{5}$\TBstrut\\
	$SU(4)$ & $SO(4)$ & $72$ & $720$ & $5760$ &
	$\frac{2\sqrt{5}}{15}$ & $\frac{6\sqrt{5}}{5}$\TBstrut\\
	\hline
	\end{tabular}
\end{center}
\end{table}
\noindent
Note that, in every table, we excluded all the space forms appearing in the
classification of classical symmetric spaces (for instance,
$SO(n+1)/SO(n)\cong\SS^n$). Moreover,
\begin{itemize}
\item
$U(n)/(U(1)\times U(n))$, where $n\in\NN$,
can be regarded as $\cp^{n}$; also, $SO(6)/U(3)$ is
$\cp^3$, while $SO(4)\cong \SS^3\times\rp^3$ has
$SU(2)\times SU(2)\cong \SS^3\times\SS^3$ as its universal cover;
\item there are no $7$-dimensional classical irreducible symmetric spaces
which are not space forms.
\end{itemize}
We can also obtain locally symmetric spaces by taking into account
Cartesian products $M\times N$ of irreducible symmetric Einstein manifolds,
with the product metric $g=g_M+\beta g_N$, where $\beta$ is chosen
in such a way that $g$ is also an Einstein metric, which is unique,
up to rescaling. Exploiting
the computations in \cite{grayvanhecke} again, we derive the following
tables, where manifolds are listed up to quotients:
\begin{table}[H]
	\caption{}
	\begin{center}
		\begin{tabular}{|p{2.5cm} |p{1.5cm} |p{1.5cm} |p{1.5cm} |p{1.5cm} |p{1.5cm} |}
			\hline
			\multicolumn{6}{|c|}{$5$-dimensional symmetric Einstein product spaces}
			\\
			\hline
			Type & $S$ & $\abs{\weyl}^2$ & $Q$ & $C_M$ & $A_M$\TBstrut\\
			\hline
			$\SS^2\times\SS^3$ & $5$ & $\frac{9}{2}$ & $\frac{9}{2}$ &
			$\frac{\sqrt{2}}{3}$ & $\frac{5\sqrt{2}}{3}$
			\TBstrut\\
			\hline
		\end{tabular}
	\label{5product}
	\end{center}
\end{table}

\begin{table}[H]
	\caption{}
	\begin{center}
		\begin{tabular}{|p{2.5cm} |p{1.5cm} |p{1.5cm} |p{1.5cm} |p{1.5cm} |p{1.5cm} |}
			\hline
			\multicolumn{6}{|c|}{$6$-dimensional symmetric Einstein product spaces}
			\\
			\hline
			Type & $S$ & $\abs{\weyl}^2$ & $Q$ & $C_M$ & $A_M$\TBstrut\\
			\hline
			$\SS^2\times\SS^4$ & $24$ & $\frac{1024}{15}$ & $\frac{4096}{15}$ &
			$\frac{\sqrt{15}}{8}$ & $\frac{3\sqrt{15}}{4}$
			\TBstrut\\
				$\SS^2\times\cp^2$ & $6$ & $\frac{104}{15}$ & $\frac{104}{15}$
			& $\frac{\sqrt{390}}{52}$ & $\frac{3\sqrt{390}}{26}$
			\TBstrut\\
			$\SS^2\times\SS^2\times\SS^2$ & $12$ & $\frac{192}{5}$ &
			$\frac{384}{5}$ & $\frac{\sqrt{15}}{12}$ & $\frac{\sqrt{15}}{2}$
			\TBstrut\\
			$\SS^3\times\SS^3$ & $48$ & $\frac{1152}{5}$ & $\frac{9216}{5}$
			& $\frac{\sqrt{10}}{6}$ & $\sqrt{10}$ \TBstrut\\
			\hline
		\end{tabular}
	\end{center}
\end{table}

\begin{table}[H]
	\caption{}
	\begin{center}
		\begin{tabular}{|p{3.5cm} |p{1.5cm} |p{1.5cm} |p{1.5cm} |p{1.5cm} |p{1.5cm} |}
			\hline
			\multicolumn{6}{|c|}{$7$-dimensional symmetric Einstein product spaces}
			\\
			\hline
			Type & $S$ & $\abs{\weyl}^2$ & $Q$ & $C_M$ & $A_M$\TBstrut\\
			\hline
			$\SS^3\times\SS^4$ & $56$ & $\frac{640}{3}$ & $\frac{5120}{3}$ &
			$\sqrt{\frac{3}{10}}$ & $\frac{7\sqrt{30}}{10}$
			\TBstrut\\
			$\SS^3\times\cp^2$ & $14$ & $24$ & $48$
			& $\frac{\sqrt{6}}{6}$ & $\frac{7\sqrt{6}}{6}$
			\TBstrut\\
			$\SS^3\times\SS^2\times\SS^2$ & $14$ & $\frac{104}{3}$ &
			$\frac{208}{3}$ & $\sqrt{\frac{3}{26}}$ & $7\sqrt{\frac{3}{26}}$
			\TBstrut\\
			$\SS^2\times\SS^5$ & $7$ & $\frac{25}{6}$ & $\frac{25}{6}$
			& $\frac{\sqrt{6}}{5}$ & $\frac{7\sqrt{6}}{5}$ \TBstrut\\
			$\SS^2\times(\bigslant{SU(3)}{SO(3)})$ & $14$ & $40$
			& $80$ & $\frac{\sqrt{10}}{10}$ & $\frac{7\sqrt{10}}{10}$
			\TBstrut\\
			\hline
		\end{tabular}
	\end{center}
\end{table}

\begin{table}[H]
	\caption{}
	\begin{center}
		\begin{tabular}{|p{4.5cm} |p{1.5cm} |p{1.5cm} |p{1.5cm} |p{1.5cm} |p{1.5cm} |}
			\hline
			\multicolumn{6}{|c|}{$8$-dimensional symmetric Einstein product spaces}
			\\
			\hline
			Type & $S$ & $\abs{\weyl}^2$ & $Q$ & $C_M$ & $A_M$\TBstrut\\
			\hline
			$\SS^4\times\SS^4$ & $48$ & $\frac{192}{7}$ & $\frac{576}{7}$ &
			$\frac{\sqrt{21}}{8}$ & $\sqrt{21}$
			\TBstrut\\
			$\SS^4\times\cp^2$ & $72$ & $\frac{360}{7}$ & $\frac{1080}{7}$
			& $\frac{\sqrt{70}}{20}$ & $\frac{2\sqrt{70}}{5}$
			\TBstrut\\
			$\SS^4\times\SS^2\times\SS^2$, $\cp^2\times\cp^2$ & $24$ & $\frac{528}{7}$ & $\frac{1584}{7}$
			& $\frac{\sqrt{231}}{44}$ & $\frac{2\sqrt{231}}{11}$
			\TBstrut\\
			$\cp^2\times\SS^2\times\SS^2$ & $24$ & $\frac{696}{7}$ & $\frac{2088}{7}$
			& $\frac{\sqrt{1218}}{116}$ & $\frac{2\sqrt{1218}}{29}$ \TBstrut\\
			$\SS^2\times\SS^2\times\SS^2\times\SS^2$ & $16$ & $\frac{384}{7}$ & $\frac{768}{7}$
			& $\frac{\sqrt{42}}{24}$ & $\frac{\sqrt{42}}{3}$
			\TBstrut\\
			$\SS^3\times\SS^5$ & $16$ & $\frac{90}{7}$ & $\frac{180}{7}$ &
			$\frac{\sqrt{70}}{15}$ & $\frac{8\sqrt{70}}{15}$ \TBstrut\\
			$\SS^2\times\SS^3\times\SS^3$, $\SS^2\times\cp^3$ &
			$8$ & $\frac{54}{7}$ & $\frac{54}{7}$ & $\frac{\sqrt{42}}{18}$
			& $\frac{4\sqrt{42}}{9}$ \TBstrut\\
			$\SS^3\times(\bigslant{SU(3)}{SO(3)})$ & $16$ & $\frac{760}{21}$
			& $\frac{1520}{21}$ & $\sqrt{\frac{21}{190}}$ &
			$4\sqrt{\frac{42}{95}}$ \TBstrut\\
			\hline
		\end{tabular}
	\end{center}
\end{table}

\begin{table}[H]
	\caption{}
\begin{center}
	\begin{tabular}{|p{4.5cm} |p{1.5cm} |p{1.5cm} |p{1.5cm} |p{1.5cm} |p{1.5cm} |}
		\hline
		\multicolumn{6}{|c|}{$9$-dimensional symmetric Einstein product spaces}
		\\
		\hline
		Type & $S$ & $\abs{\weyl}^2$ & $Q$ & $C_M$ & $A_M$\TBstrut\\
		\hline
		$\SS^5\times\SS^4$ & $36$ & $\frac{140}{3}$ & $\frac{560}{3}$ &
		$2\sqrt{\frac{3}{35}}$ & $18\sqrt{\frac{3}{35}}$
		\TBstrut\\
		$\SS^5\times\cp^2$ & $36$ & $\frac{268}{3}$ & $\frac{1072}{3}$
		& $2\sqrt{\frac{3}{67}}$ & $18\sqrt{\frac{3}{67}}$
		\TBstrut\\
		$\SS^5\times\SS^2\times\SS^2$ & $36$ & $132$ &
		$628$ & $\frac{2\sqrt{33}}{33}$ & $\frac{6\sqrt{33}}{11}$
		\TBstrut\\
		$(\SS^2\times\SS^3)\times(\SS^2\times\SS^2)$ & $9$ & $\frac{51}{4}$ & $\frac{51}{4}$
		& $\frac{2\sqrt{51}}{51}$ & $\frac{6\sqrt{51}}{17}$ \TBstrut\\
		$\SS^2\times\SS^3\times\SS^4$ & $9$ & $\frac{89}{12}$
		& $\frac{89}{12}$ & $2\sqrt{\frac{3}{89}}$ & $18\sqrt{\frac{3}{89}}$
		\TBstrut\\
		$\SS^2\times\SS^3\times\cp^2$ & $9$ & $\frac{121}{12}$
		& $\frac{121}{12}$ & $\frac{2\sqrt{3}}{11}$ & $\frac{18\sqrt{3}}{11}$
		\TBstrut\\
		$(\bigslant{SU(3)}{SO(3)})\times\SS^2\times\SS^2$ & $54$ & $507$
		& $3042$ & $\frac{2\sqrt{3}}{13}$ & $\frac{18\sqrt{3}}{13}$
		\TBstrut\\
		$(\bigslant{SU(3)}{SO(3)})\times\SS^4$ & $54$ & $315$
		& $1890$ & $\frac{2\sqrt{35}}{35}$ & $\frac{18\sqrt{35}}{35}$
		\TBstrut\\
		$(\bigslant{SU(3)}{SO(3)})\times\cp^2$ & $54$ & $411$
		& $2466$ & $\frac{2\sqrt{411}}{137}$ & $\frac{18\sqrt{411}}{137}$
		\TBstrut\\
		$\SS^3\times\SS^3\times\SS^3$, $\SS^3\times\cp^3$ &
		$72$ & $432$ & $3456$ & $\dfrac{2\sqrt{3}}{9}$ & $2\sqrt{3}$
		\TBstrut\\
		\hline
	\end{tabular}
\end{center}
\end{table}
Note that, comparing the Tables of irreducible spaces and product
manifolds, we get that the maximal constants $C_M$ and $A_M$
for $5\leq n\leq 9$
are given by
\begin{itemize}
	\item $\SS^{\frac{n-1}{2}}\times\SS^{\frac{n-1}{2}+1}$ with the standard
	product metric for $n=5,7,9$;
	\item $\cp^3$ with the Fubini-Study metric and $\SS^3\times\SS^3$
	with the standard product metric for $n=6$;
	\item $\SS^4\times\SS^4$ with the standard product metric for $n=8$.
\end{itemize}
\subsection{A lower bound in the six-dimensional case}
In this section, we provide a lower bound for the optimal constant
$C(6)$ which realizes \eqref{qestimate}. Recall that, in dimension $4$,
the equality in \eqref{qestimate} is achieved by $\cp^2$, endowed
with the Fubini-Study metric $g_{FS}$. However, this does not hold in higher
dimensional cases in general: in fact, we prove that, if $n=6$, the equality
in \eqref{qestimate} cannot realized by a symmetric space.

We recall that, if $(M,g)$ is an oriented four-dimensional Riemannian
manifold, the \emph{twistor space} $Z$ associated to $(M,g)$
(\cite{athisin}) can be defined as
the set of all pairs $(p,J)$, where $p\in M$ and $J$ is an orthogonal
complex structure on the tangent space $T_pM$: alternatively,
one can consider the representation of the group $U(2)$ in $SO(4)$
and define $Z$ as the $(SO(4)/U(2))$-bundle
\[
Z=O(M)_-\times_{SO(4)}\bigslant{SO(4)}{U(2)},
\]
where $O(M)_-$ is the negatively oriented orthonormal frame bundle of $M$.
More clearly, since $O(M)_-\longra M$ is a principal $SO(4)$-bundle,
$Z$ can be regarded as the associated fiber bundle: therefore,
standard theory of principal bundles (\cite{kobnom}) allows us to define the twistor
space as
\[
Z=\bigslant{O(M)_-}{U(2)}
\]
and, therefore, as the sphere bundle of $2$-forms in $\Lambda_-$ of norm
$\sqrt{2}$
(for a more complete dissertation about this construction of twistor
spaces, see, for instance, \cite{debnan}, \cite{jensenrigoli}
and \cite{salamon}.
We also refer the reader to the useful surveys \cite{davmus} and
\cite{lebruntwist} and the references therein). There exists a
$1$-parameter family of Riemannian metrics $g_t$ on $Z$, where $t>0$,
defined as the pullback of Riemannian metrics $h_t$ on $O(M)_-$ \emph{via}
the $U(2)$-bundle $\sigma:O(M)_-\longra Z$ so that $\sigma$ is a Riemannian
submersion with totally geodesic fibers. Moreover, $(Z,g_t)$ becomes
an almost Hermitian manifold, since it can be endowed with two
almost complex structures $J_+$ (\cite{athisin}) and $J_-$ (\cite{eells}).

It can be shown that, when $M=\SS^4$, its twistor space $Z$ is in fact
$\cp^3$ (\cite{athisin}). It is well-known that, if we fix $t>0$ and we
consider the rescaling of the round metric $\frac{1}{t^2}g_{\SS^4}$ on
$\SS^4$, $(Z,g_t,J_+)$ is a K\"{a}hler-Einstein manifold (\cite{frikurke}),
where $g_t$ happens to be the Fubini-Study metric,
up to rescaling: this case has already been covered in Table \ref{shu},
since $(Z,g_t,J_+)$ is a compact, irreducible symmetric space. However,
by choosing $\frac{1}{2t^2}g_{\SS^4}$ on $\SS^4$, we obtain another Einstein
metric $g_t$ on $Z=\cp^3$ (\cite{frigrune}), which is the so-called
\emph{squashed metric} $g_{sq}$ (see \cite{besse},
\cite{wangziller} and \cite{ziller}).
In this setting, $(Z,g_t,J_+)$ is a strictly
nearly K\"{a}hler manifold with positive holomorphic bisectional curvature
(\cite{davmusali}) and it also is a $3$-symmetric space (for the
classification of homogeneous nearly K\"{a}hler manifold, see
\cite{butru}); however, it can be shown that $Z$ is not symmetric,
since the only locally symmetric twistor spaces associated to
half conformally flat manifolds are $\cp^3$ with the standard
K\"{a}hler-Einstein structure and $M\times\SS^2$, if
$M$ is flat (\cite{cdmtwist}).

An explicit expression for the non-zero components of the Weyl tensor of
$(Z,g_t)=(\cp^3,g_{sq})$,
viewed as the twistor space of $(\SS^4,\frac{1}{2t^2}
g_{\SS^4})$, can be obtained by the formulas listed in the appendix B
of \cite{cdmtwist}, with respect to a local orthonormal frame. Choosing
$t=1$, we get:
\begin{align} \label{weylcompsquash}
	W_{1212}&=W_{3434}=\dfrac{1}{4}, \quad W_{1234}=\dfrac{1}{8},\\
	W_{1313}&=W_{4242}=W_{1414}=W_{2323}=\dfrac{1}{16}, \quad
	W_{1342}=W_{1423}=-\dfrac{1}{16}, \notag\\
	W_{a5b5}&=W_{a6b6}=-\dfrac{3}{16}\delta_{ab}, \mbox{ for } a,b=1,...,4,
	\notag\\
	W_{1526}&=W_{3546}=-W_{1625}=-W_{3645}=\dfrac{3}{16},\notag\\
	W_{1256}&=W_{3456}=\dfrac{3}{8}, \quad
	W_{5656}=\dfrac{3}{4}\notag.
\end{align}
First, it is easy to observe that, although $Z$ is not symmetric, the
squared norm of the Weyl tensor $\abs{\weyl}^2$ is a constant function equal
to $\frac{15}{2}$: since the scalar curvature of $Z$ is constant and
equal to $\frac{15}{2}$, by \eqref{qdef} we have
\[
C_M=\dfrac{Q_{sq}}{\abs{\weyl}_{sq}^3}=\sqrt{\dfrac{3}{10}}>\dfrac{\sqrt{10}}{6},
\]
where the right-hand side is the value of $C_M$ for $(\cp^3,g_{FS})$ and
$(\SS^3\times\SS^3,\alpha g_{\SS^3}+\beta g_{\SS^3})$, where
$\alpha,\beta>0$, as it is shown
in Table \ref{shu}.
Hence, if we define
\[
C_{min}=\inf\{C\in\RR: Q'\leq C\abs{\weyl'}^3\, \mbox{ for every
algebraic Weyl curvature tensor } \weyl'\},
\]
we can conclude that
\[
C_{min}\geq\sqrt{\dfrac{3}{10}}.
\]
Although this counterexample shows that irreducible symmetric $6$-spaces do
not achieve the equality in \eqref{qestimate}, the same may be not true
for the constant $A_M$: indeed, we have
\[
A_{sq}=\dfrac{S_{sq}}{\abs{\weyl}_{sq}}=\sqrt{\dfrac{15}{2}}<
\sqrt{10},
\]
where the right-hand side is the value of $A_M$ for $(\cp^3,g_{FS})$ and
$(\SS^3\times\SS^3,\alpha g_{\SS^3}+\beta g_{\SS^3})$. Coupling this
observation with the optimal result obtained by Bour and Carron
\cite[Theorem C]{bourcarron}, we may guess that, given a closed
(conformally) Einstein $6$-manifold, the following integral pinching holds:
\[
Y(M,[g])\leq\sqrt{10}\pa{\int_M\abs{\weyl}^3d\mu_g}^{\frac{1}{3}},
\]
where equality holds if and only if $(M,g)$ is $\cp^3$ with the Fubini-Study
metric or $\SS^3\times\SS^3$ endowed with the product metric
$\alpha g_{\SS^3}+\beta g_{\SS^3}$, up to quotients.
However, we know that, in general, the estimate
\[
\int_M S^3\,d\mu_g\leq10\sqrt{10}\int_M\abs{\weyl}^3d\mu_g
\]
is not true: indeed, if, for instance,
$(M,g)=(\SS^2\times\SS^4, g_{\SS^2}+\beta g_{\SS^4})$,
where
\[
\beta>\dfrac{\sqrt{15}-3\sqrt{10}}{3\sqrt{10}-6\sqrt{15}},
\]
one can immediately observe that the opposite estimate holds (we highlight
the fact that such a manifold is locally symmetric, but not Einstein).

\subsection{A numerical approach for the sharp constant in \eqref{qestimate}}

The classical Lagrange multiplier exploited to find the optimal constant
in \eqref{qestimate} when $n=4$ is rather hard to
extend to higher-dimensional cases, due to the rapidly increasing number of
independent variables in the linear system. Therefore, we decided to
reproduce these computations \emph{via} a numerical method
\footnote{The algorithm is available under request, by sending an e-mail to
any of the authors.}, in order to
obtain a reasonable guess for the value of the constants
$C(n)$.

Our approach is the following: first, we define the Weyl tensor as a
vector $\weyl\in\RR^{n^4}$. In order to do so, we construct
the vector by labeling the components $W_{ijkl}$ as
$\weyl[x]$, where
\begin{equation} \label{weylvector}
x=(i-1)\cdot n^3+(j-1)\cdot n^2+(k-1)\cdot n+(l-1);
\end{equation}
at this point, $i,j,k,l=1,...,n$, without any symmetry condition on
$\weyl$.
Then we construct the linear constraints given
by the well-known symmetries of the Weyl curvature tensors, i.e.,
skew-symmetry with respect to the first two and the last two indices,
the first Bianchi identity and the totally trace-free condition. Hence,
the constraints are encoded in the rows of a matrix $A$ with $n^4$
columns; we recall that, given any algebraic Weyl curvature tensor
$\weyl'$, the
number of independent components of $\weyl'$ is
$m=\frac{n(n+1)(n+2)(n-3)}{12}$, therefore $A\in\ M_{n-m,n}(\RR)$.
After that, we define the function
\begin{align*}
f:\RR^{n^4}&\longra\RR\\
\weyl&\longmapsto2W_{pqrs}W_{ptru}W_{qtsu}+
\dfrac{1}{2}W_{pqrs}W_{pqtu}W_{rstu},
\end{align*}
writing every component of $\weyl$ as in \eqref{weylvector}. After
defining $\abs{\weyl}$ as usual and setting an upper and a lower bound
for the entries of $\weyl[x]$ (i.e., for instance, $\weyl[x]
\in [-1,1]$ for every $x$ defined as in \eqref{weylvector}),
we minimize the function
$-f(\weyl)/\abs{\weyl}^3$, using the Sequential Least Squared Programming
(SLSQP), an iterative method which, starting at a random vector $\weyl_0
\in\RR^{n^4}$, after some iterations gives a numerical estimate of the
maximum point of the function,
also providing an approximation of the maximum. Namely, we are able
to obtain a numerical estimate of the following quantity:
\[
\min_{\weyl\in\RR^{n^4}}-\dfrac{f(\weyl)}{\abs{\weyl}^3},
\]
under the constraints given by $A\cdot \weyl=0$ and $\weyl[x]\in [-1,1]$.
However, due to the heavy computational cost, we could not manage to
perform this Lagrange multiplier argument if $n>6$ for
now;
on the other hand, we verified the correctness of the algorithm,
by recovering the sharp constant $\frac{\sqrt{6}}{4}$ in dimension four.

After many attempts, also starting from many different initial data, the algorithm hints that the sharp constant in dimension $5$ and $6$ might
be the same as in dimension $4$: namely,
\[
C(4)=\dfrac{\sqrt{6}}{4} \mbox{ and } C(5),C(6)\approx \dfrac{\sqrt{6}}{4}.
\]
We also checked the convergence of the algorithm, using standard numerical
analysis arguments, in order to verify the effectiveness of our procedure:
an example is given in Figure \ref{figura} below, where it is apparent that,
starting from different initial random vectors, the error converges to
zero.

\begin{center}
	\begin{minipage}{0.7\linewidth}
		\includegraphics[width=\linewidth]{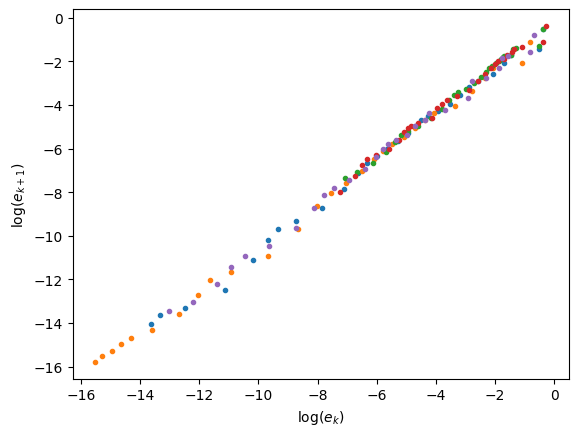}
	\end{minipage}
	\captionof{figure}{Estimates for the order of convergence of
		$\abs{\dfrac{f(\weyl)}{\abs{\weyl}^3}-\dfrac{\sqrt{6}}{4}}$ for $n=5$.
		Here, $\log(e_k)=
		\log\abs{\dfrac{f(\weyl_k)}{\abs{\weyl_k}^3}-\dfrac{\sqrt{6}}{4}}$,
		where $\weyl_k$ is the iteration at the $k$-th step. The scale
		of both axes are logarithmic.
		\label{figura}}
	\hfill
\end{center}
\noindent
This numerical result leads us to conjecture that the estimate
\[
Q\leq C(n)\abs{\weyl}^3
\]
holds with $C(n)=\frac{\sqrt{6}}{4}$ for every $n$:
however, if, on one hand, the equality is achieved when $n=4$ by an algebraic
Weyl tensor which is, in fact, the actual Weyl tensor of a metric $g$ on
a smooth manifold $M$, on the other hand, in higher-dimensional cases
equality in \eqref{qestimate} might be realized by some algebraic
Weyl curvature tensor which does not derive from a Riemannian metric.

\bibliographystyle{abbrv}
\bibliography{6D_biblio}
\end{document}